\documentclass[11pt]{article} 
\usepackage{amsfonts,amsmath,latexsym,amssymb,mathrsfs,amsthm,comment}
\usepackage{slashbox}
\usepackage{caption}
\usepackage{graphicx}
\usepackage{xcolor}

\let\OLDthebibliography\thebibliography
\renewcommand\thebibliography[1]{
  \OLDthebibliography{#1}
  \setlength{\parskip}{3pt}
  \setlength{\itemsep}{0pt plus 0.3ex}
}

\def\numberlikeadb{\global\def\theequation{\thesection.\arabic{equation}}}
\numberlikeadb
\newtheorem{theorem}{Theorem}[section]

\newtheorem{corollary}[theorem]{Corollary}

\newtheorem{proposition}[theorem]{Proposition}
\newtheorem{remark}[theorem]{Remark}

\usepackage{lscape}
\usepackage{caption}
\usepackage{multirow}
\allowdisplaybreaks
\begin{document}

\title{The variance-gamma ratio distribution
}
\author{Robert E. Gaunt\footnote{Department of Mathematics, The University of Manchester, Oxford Road, Manchester M13 9PL, UK, robert.gaunt@manchester.ac.uk; siqi.li-8@postgrad.manchester.ac.uk}\:\, and Siqi L$\mathrm{i}^{*}$}

\date{} 
\maketitle

\vspace{-10mm}

\begin{abstract} Let $X$ and $Y$ be independent variance-gamma random variables with zero location parameter; then the exact probability density function of the ratio $X/Y$ is derived. Some basic distributional properties are also derived, including identification of parameter regimes under which the density is bounded, asymptotic approximations of tail probabilities, and fractional moments; in particular, we see that the mean is undefined. In the case that $X$ and $Y$ are independent symmetric variance-gamma random variables, an exact formula is also given for the cumulative distribution function of the ratio $X/Y$.
\end{abstract}


\noindent{{{\bf{AMS 2010 Subject Classification:}}} Primary 60E05; 62E15

\section{Introduction}

The variance-gamma (VG) distribution with parameters $m > -1/2$, $0\leq|\beta|<\alpha$, $\mu \in \mathbb{R}$ has probability density function (PDF)
\begin{equation}\label{vgpdf} f_X(x) = M \mathrm{e}^{\beta (x-\mu)}|x-\mu|^{m}K_{m}(\alpha|x-\mu|), \quad x\in \mathbb{R},
\end{equation}
where the normalising constant is given by
\[M=M_{m,\alpha,\beta}=\frac{\gamma^{2m+1}}{\sqrt{\pi}(2\alpha)^m \Gamma(m+1/2)},\]
with $\gamma^2=\alpha^2-\beta^2$. Here
 $K_m(x)$ is a modified Bessel function of the second kind, which is defined in Appendix \ref{appa}. If the random variable $X$ has PDF (\ref{vgpdf}), then we write $X\sim\mathrm{VG}(m,\alpha,\beta,\mu)$. The VG distribution is also referred to as the Bessel function distribution \cite{m32}, the generalized Laplace distribution \cite{kkp01} and the McKay Type II distribution \cite{ha04}. Alternative parametrisations can be found in \cite{gaunt vg,kkp01,mcc98}.  The VG distribution is widely used in financial modelling \cite{mcc98,madan}, 
and further application areas and distributional properties are given in
Chapter 4 of the book \cite{kkp01}. 
In this paper, we set $\mu=0$.

Let $X$ and $Y$ be independent random variables. The distribution of the ratio $X/Y$ for $X\sim\mathrm{VG}(m,0,1,0)$ and $Y\sim\mathrm{VG}(n,0,1,0)$, $m,n>1$, was studied by \cite{nk06} (they referred to $X$ and $Y$ as Bessel function random variables). In particular, an exact formula, expressed in terms of the Gaussian hypergeometric function (defined in Appendix \ref{appa}), was obtained for the PDF. The study of \cite{nk06} was motivated by the fact that the ratio of independent random variables arises throughout the sciences and the already mentioned fact that VG random variables appear in many applications, so that the ratio $X/Y$ could, for example, represent the ratio of distributions of log-returns of two different financial assets. However, it has been observed, for example, that when fitting the VG distribution to log-returns of financial assets the skewness parameter $\beta$ is non-zero (see, for example, \cite{s04}). This provides motivation for this paper, in which we derive the exact distribution of the VG ratio $X/Y$ for independent $X\sim \mathrm{VG}(m,\alpha_1,\beta_1,0)$ and $Y\sim\mathrm{VG}(n,\alpha_2,\beta_2,0)$, with $m,n>-1/2$, $0\leq|\beta_i|<\alpha_i$, $i=1,2$, thereby generalising results of \cite{nk06}. In the case of a ratio of independent symmetric VG random variables ($\beta_1=\beta_2=0$), we also derive an exact formula for the cumulative distribution function (CDF) of the ratio $X/Y$, a key distributional property that is not given in \cite{nk06}. This formula is expressed in terms of the Meijer $G$-function, which is defined in Appendix \ref{appa}.

We stress that our results hold for $m,n>-1/2$, a wider range of validity than claimed by \cite{nk06}. This is significant because the regime $m<1$ is often encountered in applications; for example, when fitting the VG distribution to log returns of financial assets \cite{s04}. Another source of interest is that in the case $m=0$ the VG distribution corresponds to the distribution of the product of two correlated zero mean normal random variables, which itself has numerous applications dating back to 1936 with the work of \cite{craig}; see \cite{gaunt22} for an overview. Results corresponding to this case are given in Corollary \ref{cor2.10}.
Moreover, the case $m,n<1$ warrants attention because the VG ratio distribution undergoes a significant change in behaviour in that the density is bounded for $m>0$, but has a singularity at the origin for $-1/2<m\leq0$;  see Proposition \ref{prop2.50}. The tails of the distribution also become heavier if $-1/2<n\leq0$; see Proposition \ref{prop2.5} and Corollary \ref{cor2.7}, where in the corollary we give asymptotic approximations for tail probabilities. 
It follows from Proposition \ref{prop2.5} that the mean of the VG ratio distribution is undefined.

\section{The variance-gamma ratio distribution}

\subsection{Probability density function}


In the following theorem, we provide an explicit expression for the PDF $f_Z(z)$ of $Z=X/Y$ in terms of an infinite series involving the Gaussian hypergeometric function.
Throughout this paper, we let $\gamma_i^2=\alpha_i^2-\beta_i^2$, $i=1,2$. We also denote $a_{ij}=(1+(-1)^{i+j})/2$, $i,j\geq0$, so that $a_{ij}=1$ if $i+j$ is even, and $a_{ij}=0$ if $i+j$ is odd. 

\begin{theorem}\label{thm1}
Let $m,n>-1/2$, $0\leq|\beta_i|<\alpha_i$, $i=1,2$. Suppose $X \sim \mathrm{VG}(m,\alpha_1,\beta_1,0) $ and $Y \sim \mathrm{VG}(n,\alpha_2,\beta_2,0) $ are independent. Let $Z=X/Y$. Then, for $z\in\mathbb{R}$, 
\begin{align}\label{eq:3}
f_Z(z) & = \frac{\gamma_1^{2m+1}\gamma_2^{2n+1}|z|^{-2n-2}}{\pi\alpha_1^{2m+2n+2}\Gamma(m+1/2)\Gamma(n+1/2)}  \sum^\infty_{i,j=0} \frac{ \beta_1^i\beta_2 ^j}{i!j!} \frac{2^{i+j} a_{ij}((i+j)/2)!z^{-j} }{\alpha_1^{i+j} \Gamma(m+n+2+i+j)}  \nonumber
 \\ &\quad \times \Gamma\Big(m+n+1+\frac{i+j}{2}\Big)\Gamma\Big(m+1+\frac{i+j}{2}\Big) \Gamma\Big(n+1+\frac{i+j}{2}\Big) \nonumber
 \\ &\quad \times {}_2F_1\bigg(m+n+1+\frac{i+j}{2},n+1+\frac{i+j}{2};m+n+2+i+j; 1-\frac{\alpha^2_2}{\alpha_1^2 z^2} \bigg).
\end{align}
\end{theorem}

\begin{remark}When $\beta_1=\beta_2=0$, the PDF $f_Z(z)$ can be expressed as a single hypergeometric function:
\begin{equation}\label{mnzero}f_Z(z)=\frac{\alpha_2^{2n+1}|z|^{-2n-2}\Gamma(m+1)\Gamma(n+1)}{\pi\alpha_1^{2n+1}(m+n+1)\Gamma(m+1/2)\Gamma(n+1/2)}{}_2F_1\bigg(m+n+1,n+1;m+n+2; 1-\frac{\alpha^2_2}{\alpha_1^2 z^2} \bigg),
\end{equation}
in agreement with the formula given in Theorem 2.1 of \cite{nk06}. We note that an application of the formula (\ref{hlog}) yields the following further simplification when $m=n=0$:
\begin{equation}\label{hlog2}f_Z(z)=\frac{2\alpha_1}{\pi^2 \alpha_2}\frac{\log|\alpha_1z/\alpha_2|}{(\alpha_1z/\alpha_2)^2-1}.
\end{equation}
When $\beta_1=0$ the PDF $f_Z(z)$ simplifies to
\begin{align}
f_Z(z) & = \frac{\gamma_2^{2n+1}|z|^{-2n-2}}{\pi\alpha_1^{2n+1}}  \sum^\infty_{k=0} \frac{k!}{(2k)!} \frac{\Gamma(m+n+1+k)\Gamma(m+1+k) \Gamma(n+1+k)}{\Gamma(m+1/2)\Gamma(n+1/2)\Gamma(m+n+2+2k)}  \nonumber
 \\ &\quad \times\bigg(\frac{2\beta_2}{\alpha_1z}\bigg)^{2k} {}_2F_1\bigg(m+n+1+k,n+1+k;m+n+2+2k; 1-\frac{\alpha^2_2}{\alpha_1^2 z^2} \bigg), \nonumber
\end{align}
and
when $\beta_2=0$ the PDF simplifies to
\begin{align}
f_Z(z) & = \frac{\gamma_1^{2m+1}\alpha_2^{2n+1}|z|^{-2n-2}}{\pi\alpha_1^{2m+2n+1}}  \sum^\infty_{k=0} \frac{k!}{(2k)!}    \frac{\Gamma(m+n+1+k)\Gamma(m+1+k) \Gamma(n+1+k)}{\Gamma(m+1/2)\Gamma(n+1/2)\Gamma(m+n+2+2k)}  \nonumber
 \\ &\quad \times\bigg(\frac{2\beta_1}{\alpha_1}\bigg)^{2k} {}_2F_1\bigg(m+n+1+k,n+1+k;m+n+2+2k; 1-\frac{\alpha^2_2}{\alpha_1^2 z^2} \bigg). \nonumber
\end{align}
Observe that the PDF $f_Z(z)$ is symmetric about the origin if $\beta_1=0$ or $\beta_2=0$.

We remark that the exact PDF of the product of correlated normal random variables is expressed as a double sum involving the modified Bessel function of the second kind when both normal random variables have non-zero mean \cite{cui}, and the PDF simplifies to a single infinite series of modified Bessel functions of the second kind when one of the means is zero \cite{cui},
and simply to the VG PDF when both means are zero \cite{gaunt prod}. Thus, the increase in complexity of the PDF of the VG ratio distribution for non-zero skewness parameters $\beta_1,\beta_2$ mirrors the increase in complexity of the product of two correlated normal random variables as one moves from zero means to non-zero means.
\end{remark}

\begin{proof}
We consider the case $z>0$; the case $z<0$ is similar and thus omitted from this proof. The PDF of $Z = X/Y$ for $z>0$ can be expressed as
\begin{align*}
f_Z(z) &= \int^\infty_{-\infty} |y| f_X(yz) f_Y(y) \,\mathrm{d}y =I_1+I_2,
\end{align*}
where
\begin{align}\label{i1}I_1&=M_1 M_2 \int^\infty_{0} y^{m+n+1} \mathrm{e}^{\beta_1 yz} \mathrm{e}^{\beta_2 y} z^m K_m(\alpha_1 yz)  K_n(\alpha_2 y)\, \mathrm{d}y, \\
\label{i2}I_2&= M_1 M_2  \int^0_{-\infty} (-y)^{m+n+1} \mathrm{e}^{\beta_1 yz} \mathrm{e}^{\beta_2 y} z^m K_m(-\alpha_1 yz)  K_n(-\alpha_2 y) \,\mathrm{d}y,
\end{align}
and $M_1:=M_{m,\alpha_1,\beta_1}$ and $M_2:=M_{n,\alpha_2,\beta_2}$ are the normalising constants of the $\mathrm{VG}(m,\alpha_1,\beta_1,0)$ and $\mathrm{VG}(n,\alpha_2,\beta_2,0)$ distributions. On using Taylor's expansion of the exponential function and interchanging integration and summation in the first step, and evaluating the integral in the second step using equation (6.576.4) of \cite{gradshetyn}, we obtain that 
\begin{align}\label{eq:5}
 I_1 & =  M_1 M_2 z^m \sum^\infty_{i=0} \sum^\infty_{j=0} \frac{\beta_1^i z^{i}}{i!}  \frac{\beta_2^j}{j!} \int^\infty_{0}  y^{m+n+1+i+j}  K_m(\alpha_1 yz)  K_n(\alpha_2 y)\, \mathrm{d}y \nonumber
 \\ & = \frac{ M_1 M_2 2^{m+n-1} \alpha_2^n}{\alpha_1^{m+2n+2}z^{2n+2}}  \sum^\infty_{i=0} \sum^\infty_{j=0} \frac{\beta_1^i}{i!}  \frac{\beta_2^j}{j!} \frac{2^{i+j} }{\alpha_1^{i+j}z^j\Gamma(m+n+2+i+j)} \Gamma\Big(m+n+1+\frac{i+j}{2}\Big)   \nonumber
 \\ & \quad\times \Gamma\Big(n+1+\frac{i+j}{2}\Big) \Gamma\Big(m+1+\frac{i+j}{2}\Big) \Gamma\Big(1+\frac{i+j}{2}\Big)  \nonumber
 \\ & \quad\times {}_2F_1\bigg(m+n+1+\frac{i+j}{2},1+n+\frac{i+j}{2};m+n+2+i+j; 1-\frac{\alpha^2_2}{\alpha_1^2 z^2} \bigg).
\end{align}
By a similar calculation, we get that
\begin{align}
\label{eq:4}
I_2 &= \frac{ M_1 M_2  2^{m+n-1} \alpha_2^n}{\alpha_1^{m+2n+2} z^{2n+2}}  \sum^\infty_{i=0} \sum^\infty_{j=0} \frac{ (-1)^{i+j} \beta_1^i\beta_2 ^j}{i!j!} \frac{2^{i+j} }{\alpha_1^{i+j}z^j\Gamma(m+n+2+i+j)} \nonumber
 \\ & \quad\times \Gamma\Big(m+n+1+\frac{i+j}{2}\Big) \Gamma\Big(n+1+\frac{i+j}{2}\Big) \Gamma\Big(m+1+\frac{i+j}{2}\Big) \Gamma\Big(1+\frac{i+j}{2}\Big) \nonumber
 \\ & \quad\times {}_2F_1\bigg(m+n+1+\frac{i+j}{2},n+1+\frac{i+j}{2};m+n+2+i+j; 1-\frac{\alpha^2_2}{\alpha_1^2 z^2} \bigg).
\end{align}
Summing up (\ref{eq:5}) and (\ref{eq:4}) now yields (\ref{eq:3}) for $z>0$.
\end{proof}

Using special properties of the modified Bessel function of the second kind and the hypergeometric function, equivalent forms and elementary formulas for the PDF of $Z=X/Y$ can be derived; we illustrate this in the following Corollary \ref{cor.1} and Proposition \ref{cor.2}. 

\begin{corollary}\label{cor.1}
The PDF (\ref{eq:3}) can be expressed in the equivalent forms:
\begin{align}\label{eq:6}
f_Z(z) & = \frac{\gamma_1^{2m+1}\gamma_2^{2n+1}|z|^{2m}}{\pi\alpha_2^{2m+2n+2}\Gamma(m+1/2)\Gamma(n+1/2)}  \sum^\infty_{i,j=0 } \frac{ \beta_1^i\beta_2 ^j}{i!j!} \frac{2^{i+j} a_{ij}((i+j)/2)!z^{i}  }{\alpha_2^{i+j}  \Gamma(m+n+2+i+j)} \nonumber
 \\ & \quad \times \Gamma\Big(m+n+1+\frac{i+j}{2}\Big)  \Gamma\Big(m+1+\frac{i+j}{2}\Big) \Gamma\Big(n+1+\frac{i+j}{2}\Big)  \nonumber
 \\ & \quad \times {}_2F_1\bigg(m+n+1+\frac{i+j}{2},m+1+\frac{i+j}{2};m+n+2+i+j; 1-\frac{\alpha^2_1z^2}{\alpha_2^2} \bigg), \\
f_Z(z) & =  \frac{\gamma_1^{2m+1}\gamma_2^{2n+1}}{\pi\alpha_1^{2m}\alpha_2^{2n+2}\Gamma(m+1/2)\Gamma(n+1/2)}  \sum^\infty_{i,j=0} \frac{ \beta_1^i\beta_2 ^j}{i!j!} \frac{2^{i+j} a_{ij}((i+j)/2)!)z^{i}  }{\alpha_2^{i+j}  \Gamma(m+n+2+i+j)}  \nonumber
 \\ & \quad\times \Gamma\Big(m+n+1+\frac{i+j}{2}\Big)\Gamma\Big(m+1+\frac{i+j}{2}\Big) \Gamma\Big(n+1+\frac{i+j}{2}\Big) \nonumber
\\ 
\label{eq:7}& \quad\times {}_2F_1\bigg(n+1+\frac{i+j}{2},1+\frac{i+j}{2};m+n+2+i+j; 1-\frac{\alpha^2_1z^2}{\alpha_2^2 } \bigg),
\end{align}
and
\begin{align}\label{eq:8}
f_Z(z) & = \frac{\gamma_1^{2m+1}\gamma_2^{2n+1}z^{-2}}{\pi\alpha_1^{2m+2}\alpha_2^{2n}\Gamma(m+1/2)\Gamma(n+1/2)}  \sum^\infty_{i,j=0} \frac{ \beta_1^i \beta_2^j}{i!j!} \frac{2^{i+j} a_{ij}((i+j)/2)!z^{-j}  }{\alpha_1^{i+j}  \Gamma(m+n+2+i+j)} \nonumber
 \\ & \quad\times \Gamma\Big(m+n+1+\frac{i+j}{2}\Big)\Gamma\Big(m+1+\frac{i+j}{2}\Big) \Gamma\Big(n+1+\frac{i+j}{2}\Big)  \nonumber
 \\ & \quad\times {}_2F_1\bigg(1+\frac{i+j}{2},m+1+\frac{i+j}{2};m+n+2+i+j; 1-\frac{\alpha^2_2}{\alpha_1^2 z^2} \bigg).
\end{align}
Note that (\ref{eq:6}) and (\ref{eq:7}) hold for $-\sqrt{2}\alpha_2/\alpha_1 < z < \sqrt{2}\alpha_2/\alpha_1 $, whilst (\ref{eq:8}) holds for $z < -\alpha_2/(\sqrt{2}\alpha_1)$ or $z>\alpha_2/(\sqrt{2}\alpha_1)$. 
\end{corollary} 

\begin{proof}
Apply the three transformation formulas given in equation (9.131.1) of \cite{gradshetyn} to the hypergeometric functions in the formula (\ref{eq:3}).
\end{proof}

\begin{proposition}\label{cor.2}
Suppose that $m-1/2\geq0$ and $n-1/2\geq0$ are integers. Then, (\ref{eq:3}) can be simplified to the elementary form:
\begin{align}
 f_Z(z) & = \frac{\gamma_1^{2m+1}\gamma_2^{2n+1}|z|^{m-1/2}}{(2\alpha_1)^{m+1/2}(2\alpha_2)^{n+1/2}} \sum^{m-\frac{1}{2}}_{i=0} \sum^{n-\frac{1}{2}}_{j=0}\frac{(m+n-i-j)!}{(m-1/2-i)!(n-1/2-j)!}     \binom{m-1/2+i}{i}  \nonumber
 \\\label{for1} &\quad \times \binom{n-1/2+j}{j}(2\alpha_1 |z|)^{-i} (2 \alpha_2  )^{-j}\bigg(\frac{1}{u_1^{m+n+1-i-j}}+ \frac{1}{u_2^{m+n+1-i-j}}\bigg),
\end{align}
where $u_1 =\alpha_1 |z|+\alpha_2+\beta_1 z+\beta_2$ and $ u_2 = \alpha_1 |z|+\alpha_2-\beta_1 z -\beta_2.$
\end{proposition}

\begin{remark}As noted by \cite{kkp01}, the $\mathrm{VG}(1/2,\alpha,\beta,0)$ distribution corresponds to the asymmetric Laplace distribution. Hence, setting $m=n=1/2$ in (\ref{for1}) yields the PDF of the ratio of independent asymmetric Laplace random variables. In the case $\beta_1=\beta_2=0$, the formula further reduces to the PDF of the ratio of independent Laplace random variables; this distribution is known as the Lomax distribution (see, for example, \cite{jkb}).
\end{remark}

\begin{proof}
Apply the formula (\ref{special})
to the modified Bessel functions of the second kind that appear in the integrands of $I_1$ and $I_2$ given in (\ref{i1}) and (\ref{i2}). Evaluating the resulting integrals using the standard formula $\int_0^\infty x^{a}\mathrm{e}^{-bx}\,\mathrm{d}x=a!b^{-a-1}$ for $a=0,1,2,\ldots$, $b>0$, summing up $I_1+I_2$, and simplifying yields (\ref{for1}).
\end{proof}

\subsection{Cumulative distribution function}

Let $F_Z(z)=\mathbb{P}(Z\leq z)$ denote the CDF of the VG ratio distribution. In the following theorem, we provide an exact formula for the CDF in the case of a ratio of independent symmetric VG random variables ($\beta_1=\beta_2=0$), the same setting as in the work of \cite{nk06}.

\begin{theorem}Let $\beta_1=\beta_2=0$, and $m,n>-1/2$, $\alpha_1,\alpha_2>0$. 
Then, for $z\in\mathbb{R}$,
\begin{align}\label{fgh}F_Z(z)=\frac{1}{2}+\frac{\alpha_1z/\alpha_2}{2\pi\Gamma(m+1/2)\Gamma(n+1/2)}G_{3,3}^{2,3}\bigg(\frac{\alpha_1^2z^2}{\alpha_2^2}\,\bigg|\,{-n,\frac{1}{2},0 \atop m,0,-\frac{1}{2}}\bigg).
\end{align}
\end{theorem}

\begin{proof}For ease of exposition, we set $\alpha_1=\alpha_2=1$; the general case follows from a simple rescaling. As the PDF is symmetric when $\beta_1=\beta_2=0$, we also just consider the case $z>0$. Using the representation (\ref{2f1m}) of the hypergeometric function in terms of the Meijer $G$-function allows us to write the PDF (\ref{mnzero}) in terms of the Meijer $G$-function:
\begin{equation*}f_Z(z)=C_{m,n}z^{-2n-2}G_{2,2}^{2,2}\bigg(z^{-2}\,\bigg|\,{-m-n,-n \atop 0,-n}\bigg)=C_{m,n}z^{-2n-2}G_{2,2}^{2,2}\bigg(z^{2}\,\bigg|\,{1,n+1 \atop m+n+1,n+1}\bigg), 
\end{equation*} 
where $1/C_{m,n}=\pi\Gamma(m+1/2)\Gamma(n+1/2)$, and we used the relation (\ref{meijergidentity0}) to obtain the second equality. For $z>0$, the CDF of $Z$ is thus given by
\begin{align}F_Z(z)&=\frac{1}{2}+C_{m,n}\int_0^z y^{-2n-2}G_{2,2}^{2,2}\bigg(y^{2}\,\bigg|\,{1,n+1 \atop m+n+1,n+1}\bigg)\,\mathrm{d}y\nonumber\\
&=\frac{1}{2}+\frac{C_{m,n}}{2}\int_0^{z^2} u^{-n-3/2}G_{2,2}^{2,2}\bigg(u\,\bigg|\,{1,n+1 \atop m+n+1,n+1}\bigg)\,\mathrm{d}u\nonumber\\
&=\frac{1}{2}+\frac{C_{m,n}}{2}\bigg[u^{-n-1/2}G_{3,3}^{2,3}\bigg(u\,\bigg|\,{n+\frac{3}{2},1,n+1 \atop m+n+1,n+1,n+\frac{1}{2}}\bigg)\bigg]_0^{z^2}\nonumber\\
&=\frac{1}{2}+\frac{C_{m,n}}{2}\bigg[G_{3,3}^{2,3}\bigg(u\,\bigg|\,{1,\frac{1}{2}-n,\frac{1}{2} \atop m+\frac{1}{2},\frac{1}{2},0}\bigg)\bigg]_0^{z^2}=\frac{1}{2}+\frac{C_{m,n}}{2}G_{3,3}^{2,3}\bigg(z^2\,\bigg|\,{1,\frac{1}{2}-n,\frac{1}{2} \atop m+\frac{1}{2},\frac{1}{2},0}\bigg),\nonumber
\end{align}
where we used the integral formula (\ref{mint}) in the third step; the relation (\ref{meijergidentity}) in the fourth step; and in the final step we used that the Meijer $G$-function in the penultimate equality evaluated at $u=0$ is equal to zero, which is readily deduced from the contour integral representation (\ref{mdef}) of the $G$-function. Finally, we using (\ref{meijergidentity}) yields
\begin{equation}\label{bah}F_Z(z)=\frac{1}{2}+\frac{C_{m,n}}{2}\cdot zG_{3,3}^{2,3}\bigg(z^2\,\bigg|\,{\frac{1}{2},-n,0 \atop m,0,-\frac{1}{2}}\bigg)=\frac{1}{2}+\frac{C_{m,n}}{2}\cdot zG_{3,3}^{2,3}\bigg(z^2\,\bigg|\,{-n,\frac{1}{2},0 \atop m,0,-\frac{1}{2}}\bigg).
\end{equation}
The formula (\ref{bah}) is also readily seen to hold for $z<0$. We therefore arrive at the formula (\ref{fgh}), which conveniently holds for all $z\in\mathbb{R}$
\end{proof}

\subsection{Asymptotic behaviour of the density and tail probabilities, and fractional moments}

The representations of the PDF $f_Z(z)$ given in Theorem \ref{thm1}, Corollary \ref{cor.1} and Proposition \ref{cor.2} are rather complicated and difficult to parse at first inspection. Some insight can be gained from the following propositions. 

\begin{proposition}\label{prop2.50} Let $m,n>-1/2$, $0\leq|\beta_i|<\alpha_i$, $i=1,2$. 

\vspace{2mm}

\noindent{1.} Suppose $m>0$. Then
\begin{equation*}f_Z(0)= \frac{\gamma_1^{2m+1}\gamma_2^{2n+1}}{\pi\alpha_1^{2m}\alpha_2^{2n+2}}\frac{\Gamma(m)\Gamma(n+1)}{\Gamma(m+1/2)\Gamma(n+1/2)}{}_2F_1\bigg(1,n+1;\frac{1}{2}; \frac{\beta^2_2}{\alpha_2^2 } \bigg).
\end{equation*}
\noindent{2.} Suppose $m=0$. Then, as $z\rightarrow0$,
\begin{equation*}f_Z(z)\sim-\frac{2\gamma_1\gamma_2^{2n+1}}{\pi^{3/2}\alpha_2^{2n+2}}\frac{\Gamma(n+1)\log|z|}{\Gamma(n+1/2)}{}_2F_1\bigg(1,n+1;\frac{1}{2}; \frac{\beta^2_2}{\alpha_2^2 } \bigg).
\end{equation*}

\noindent{3.} Suppose $-1/2<m<0$. Then, as $z\rightarrow0$,
\begin{equation*}f_Z(z)\sim\frac{\gamma_1^{2m+1}\gamma_2^{2n+1}}{\alpha_2^{2m+2n+2}\sin(-m\pi)}\frac{\Gamma(m+n+1)|z|^{2m}}{\Gamma(m+1/2)\Gamma(n+1/2)}{}_2F_1\bigg(m+1,m+n+1;\frac{1}{2}; \frac{\beta^2_2}{\alpha_2^2 } \bigg).
\end{equation*}
\end{proposition}

\begin{proposition}\label{prop2.5} Let $m,n>-1/2$, $0\leq|\beta_i|<\alpha_i$, $i=1,2$. 

\vspace{2mm}

\noindent{1.} Suppose $n>0$. Then, as $|z|\rightarrow\infty$,
\begin{equation}\label{l1}f_Z(z)\sim \frac{\gamma_1^{2m+1}\gamma_2^{2n+1}}{\pi\alpha_1^{2m+2}\alpha_2^{2n}}\frac{\Gamma(m+1)\Gamma(n)z^{-2}}{\Gamma(m+1/2)\Gamma(n+1/2)}{}_2F_1\bigg(1,m+1;\frac{1}{2}; \frac{\beta^2_1}{\alpha_1^2 } \bigg).
\end{equation}

\noindent{2.} Suppose $n=0$. Then, as $|z|\rightarrow\infty$,
\begin{equation}\label{l2}f_Z(z)\sim \frac{2\gamma_1^{2m+1}\gamma_2}{\pi^{3/2}\alpha_1^{2m+2}}\frac{\Gamma(m+1)z^{-2}\log|z|}{\Gamma(m+1/2)}{}_2F_1\bigg(1,m+1;\frac{1}{2}; \frac{\beta^2_1}{\alpha_1^2 } \bigg).
\end{equation}

\noindent{3.} Suppose $-1/2<n<0$. Then, as $|z|\rightarrow\infty$,
\begin{equation}\label{l3}f_Z(z)\sim\frac{\gamma_1^{2m+1}\gamma_2^{2n+1}}{\alpha_1^{2m+2n+2}\sin(-n\pi)}\frac{\Gamma(m+n+1)|z|^{-2-2n}}{\Gamma(m+1/2)\Gamma(n+1/2)}{}_2F_1\bigg(n+1,m+n+1;\frac{1}{2}; \frac{\beta^2_1}{\alpha_1^2 } \bigg).
\end{equation}
\end{proposition}

\begin{corollary}\label{cor2.7}Let $m,n>-1/2$, $0\leq|\beta_i|<\alpha_i$, $i=1,2$. Let $\bar{F}_Z(z)=\mathbb{P}(Z>z)$.

\vspace{2mm}

\noindent{1.} Suppose $n>0$. Then, as $z\rightarrow\infty$,
\begin{equation}\label{m1}\bar{F}_Z(z)\sim \frac{\gamma_1^{2m+1}\gamma_2^{2n+1}}{\pi\alpha_1^{2m+2}\alpha_2^{2n}}\frac{\Gamma(m+1)\Gamma(n)z^{-1}}{\Gamma(m+1/2)\Gamma(n+1/2)}{}_2F_1\bigg(1,m+1;\frac{1}{2}; \frac{\beta^2_1}{\alpha_1^2 } \bigg).
\end{equation}

\noindent{2.} Suppose $n=0$. Then, as $z\rightarrow\infty$,
\begin{equation*}\bar{F}_Z(z)\sim \frac{2\gamma_1^{2m+1}\gamma_2}{\pi^{3/2}\alpha_1^{2m+2}}\frac{\Gamma(m+1)z^{-1}\log(z)}{\Gamma(m+1/2)}{}_2F_1\bigg(1,m+1;\frac{1}{2}; \frac{\beta^2_1}{\alpha_1^2 } \bigg).
\end{equation*}

\noindent{3.} Suppose $-1/2<n<0$. Then, as $z\rightarrow\infty$,
\begin{equation*}\bar{F}_Z(z)\sim\frac{\gamma_1^{2m+1}\gamma_2^{2n+1}}{\alpha_1^{2m+2n+2}\sin(-n\pi)}\frac{\Gamma(m+n+1)z^{-1-2n}}{(1+2n)\Gamma(m+1/2)\Gamma(n+1/2)}{}_2F_1\bigg(n+1,m+n+1;\frac{1}{2}; \frac{\beta^2_1}{\alpha_1^2 } \bigg).
\end{equation*}
\end{corollary}


\begin{remark}Since ${}_2F_1(a,b;c;0)=1$, the formulas of Proposition \ref{prop2.50} simplify when $\beta_2=0$, whilst the formulas of Proposition \ref{prop2.5} and Corollary \ref{cor2.7} simplify when $\beta_1=0$.
\end{remark}

\begin{remark}\label{rrrr}Proposition \ref{prop2.50} tells us that, as $z\rightarrow0$, the PDF $f_Z(z)$ is bounded for $m>0$, but has a singularity for $-1/2<m\leq0$. Note that whether the density is bounded or unbounded is determined solely by the parameter $m$. We remark that this behaviour is consistent with the behaviour of the VG density $f_X(x)$, which is bounded for all $x\in\mathbb{R}$ if $m>0$, but has a singularity as $x\rightarrow0$ if $-1/2<m\leq0$ 
(see equation (2.4) of \cite{gaunt vg3}).

The asymptotic behaviour of the PDF as $z\rightarrow0$ shows that the VG ratio distribution is unimodal if $-1/2<m\leq0$, as in this case the density is bounded everywhere except for the singularity at $z=0$. Also, the VG ratio distribution was shown to be unimodal in the $\beta_1=\beta_2=0$ case by \cite{nk06} (they proved it for $m,n\geq1$, although their argument also applies for $m,n>-1/2$). 
Our numerical tests strongly suggest that the VG ratio distribution is unimodal for all parameter values, although we have been unable to prove this. 

We see from Proposition \ref{prop2.5} that the tails of the VG ratio distribution become heavier as $n\leq0$ decreases; note that now the rate of decay of the tails with respect to $z$ is solely determined by the parameter $n$. We remark that it is curious that, whilst the VG ratio distribution is only symmetric around the origin if $\beta_1=0$ or if $\beta_2=0$, for all parameter values we have $\lim_{z\rightarrow\infty}f_Z(z)/f_Z(-z)=1$.  This is in contrast to the VG distribution itself, for which $\lim_{x\rightarrow\infty}f_X(x)/f_X(-x)=1$ if and only if $\beta=0$ (see \cite[p.\ 4]{gaunt vg3}).

It follows from Proposition \ref{prop2.5} that, for all possible parameter values, the mean of the VG ratio distribution is undefined. This is a common feature of ratio distributions; for example, it is well-known that the ratio of independent standard normal random variables follows the Cauchy distribution, for which the mean does not exist. Whilst the mean of the VG ratio distribution is not defined, the fractional moments $\mathbb{E}[|Z|^k]$ do exist for suitable $k<1$, as seen in Proposition \ref{prop2.7} below.
\end{remark}

The PDF of the VG ratio distribution is plotted for a range of parameter values in Figures \ref{fig1} and \ref{fig3}. Figure \ref{fig1} demonstrates the effect of varying the shape parameters $m$ and $n$, whilst Figure \ref{fig3} shows the effect of varying the skewness parameters $\beta_1$ and $\beta_2$. 
In both figures, $\alpha_1=\alpha_2=1$. 
The figures confirm the assertions made in Remark \ref{rrrr}.


\captionsetup{font=footnotesize}

\begin{figure}\label{fig1}
\centering
\includegraphics[scale=0.47]{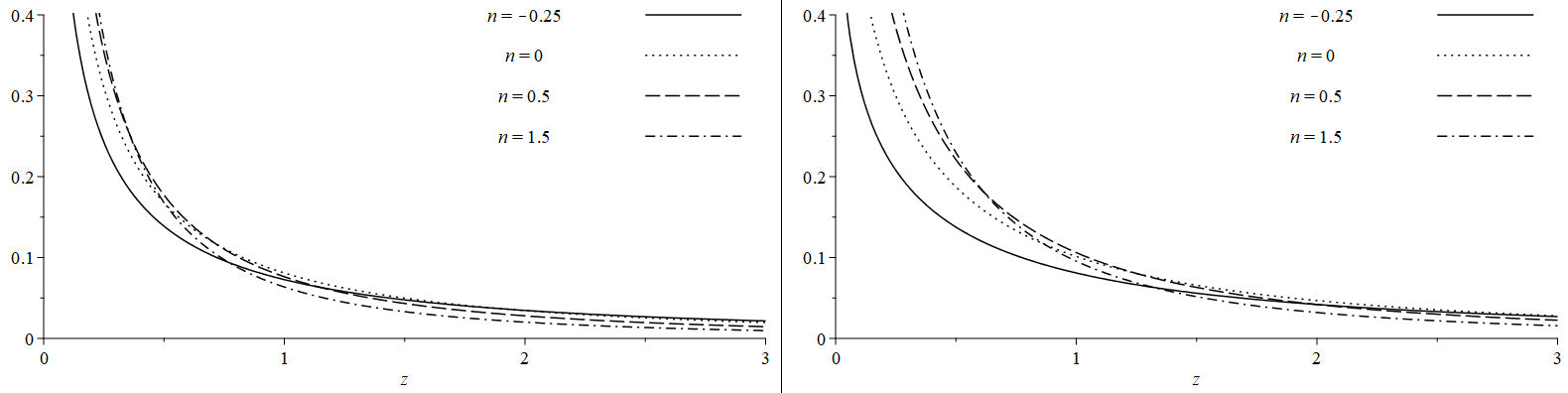}
\caption{$\beta_1=\beta_2=0$, (a): $m=-0.25$; (b): $m=0$.}
  \label{fig1}
  \end{figure}

  \begin{figure}\label{fig3}
\centering
\includegraphics[scale=0.5]{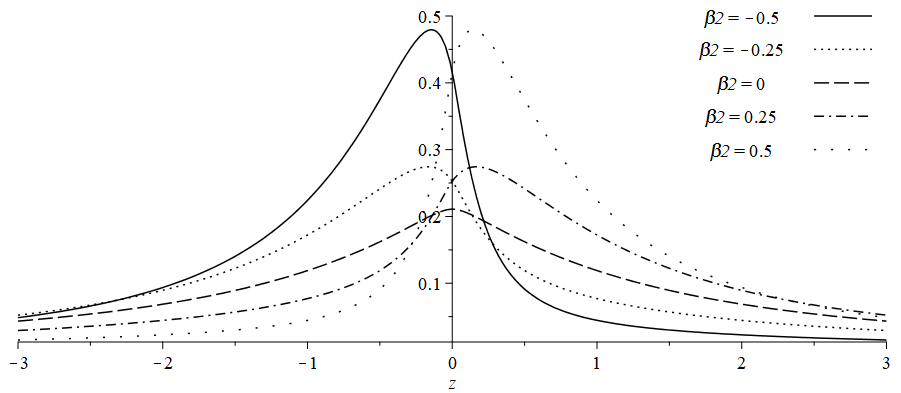}
\caption{$m=n=1.5$, $\beta_1=0.5$.}
\label{fig3}
\end{figure}


\vspace{3mm}

\noindent{\emph{Proof of Proposition \ref{prop2.50}.}} Suppose first that $m>0$. We shall make use of the representation (\ref{eq:7}) of the PDF $f_Z(z)$. Let
\[C=\frac{\gamma_1^{2m+1}\gamma_2^{2n+1}}{\pi\alpha_1^{2m}\alpha_2^{2n+2}\Gamma(m+1/2)\Gamma(n+1/2)}.\]
Then
\begin{align}\lim_{z \rightarrow 0 } f_Z(z) & =C\sum^\infty_{k=0} \bigg(\frac{2\beta_2}{\alpha_2}\bigg)^{2k}\frac{k!\Gamma(m+n+1+k)\Gamma(m+1+k)\Gamma(n+1+k)}{(2k)!\Gamma(m+n+2+2k)}\nonumber\\
\label{zeron} &\quad\times {}_2F_1(n+1+k,1+k;m+n+2+2k; 1).
\end{align}
Using the formula (\ref{f1}) to express the hypergeometric functions in (\ref{zeron}) in terms of the gamma function, and cancelling the gamma functions gives that
\begin{align*}\lim_{z \rightarrow 0 } f_Z(z) &=C\Gamma(m)\sum_{k=0}^\infty\frac{k!}{(2k)!}\Gamma(n+1+k)\bigg(\frac{2\beta_2}{\alpha_2}\bigg)^{2k}\\
&=C\Gamma(m)\Gamma(n+1){}_2F_1\bigg(1,n+1;\frac{1}{2}; \frac{\beta^2_2}{\alpha_2^2 } \bigg),
\end{align*}
where we obtained the second equality by using the basic formula $(u)_k=\Gamma(u+k)/\Gamma(u)$
to put the infinite series into the hypergeometric form (\ref{gauss}). 

The cases $m=0$ and $-1/2<m<0$ are treated similarly, but instead of using the equality (\ref{f1}) we apply the limiting forms (\ref{f2}) and (\ref{f3}) for the hypergeometric function, respectively. In the case $-1/2<m<0$, we simplify our limiting form for $f_Z(z)$ by applying the standard formula $\Gamma(-m)\Gamma(m+1)=\pi/\sin(-m\pi)$. \hfill $\Box$

\vspace{3mm}

\noindent{\emph{Proof of Proposition \ref{prop2.5}.}} The proof is similar to that of Proposition \ref{prop2.50}, except we make use of the representation (\ref{eq:8}) of the PDF $f_Z(z)$ instead of the representation (\ref{eq:7}). \hfill $\Box$

\vspace{3mm}

\noindent{\emph{Proof of Corollary \ref{cor2.7}.}} Suppose first that $n>0$. Then, by (\ref{l1}), as $z\rightarrow\infty$,
\begin{align*}\bar{F}_Z(z)\sim \frac{\gamma_1^{2m+1}\gamma_2^{2n+1}}{\pi\alpha_1^{2m+2}\alpha_2^{2n}}\frac{\Gamma(m+1)\Gamma(n)}{\Gamma(m+1/2)\Gamma(n+1/2)}{}_2F_1\bigg(1,m+1;\frac{1}{2}; \frac{\beta^2_1}{\alpha_1^2 } \bigg)\int_z^\infty x^{-2}\,\mathrm{d}x,
\end{align*} 
and evaluating $\int_z^\infty x^{-2}\,\mathrm{d}x=z^{-1}$ yields (\ref{m1}). The cases $n=0$ and $-1/2<n<0$ are dealt with similarly making use of the limiting forms (\ref{l2}) and (\ref{l3}) and the formulas $\int_z^\infty x^{-2}\log(x)\,\mathrm{d}x=z^{-1}(\log(z)+1)\sim z^{-1}\log(z)$, as $z\rightarrow\infty$, and $\int_z^\infty x^{-2-2n}\,\mathrm{d}x=z^{-1-2n}/(1+2n)$. \hfill $\Box$




\begin{proposition}\label{prop2.7}Let $m,n>-1/2$, $0\leq|\beta_i|<\alpha_i$, $i=1,2$. Then, for $\mathrm{max}\{-1,-2m-1\}<k<\mathrm{min}\{1,2n+1\}$,
\begin{align*}
 \mathbb{E}[|Z|^k] & =  \frac{ \alpha_2^{k} (1-\beta_1^2/\alpha_1^2)^{m+1/2} (1-\beta_2^2/\alpha_2^2)^{n+1/2}  }{  \alpha_1^k\cos ( k\pi/2 ) \Gamma(m+1/2) \Gamma(n+1/2) } \Gamma\Big(m+\frac{k+1}{2}\Big)\Gamma\Big(n+\frac{1-k}{2}\Big) 
 \\ &  \quad\times {}_2F_1\bigg(\frac{k+1}{2},m+\frac{k+1}{2};\frac{1}{2};\frac{\beta^2_1}{\alpha^2_1}\bigg) {}_2F_1\bigg(\frac{1-k}{2},n+\frac{1-k}{2};\frac{1}{2};\frac{\beta^2_2}{\alpha^2_2}\bigg).
\end{align*}
\end{proposition}

We will need the following formula of \cite{gauntvgmom} for the absolute moments of the $\mathrm{VG}(m,\alpha,\beta,0)$ distribution. 
Then, for $k>\mathrm{max}\{-1,-2m-1\}$,
\begin{equation}
\label{mom1}
\mathbb{E}[|X|^k]=\frac{2^k(1-\beta^2/\alpha^2)^{m+1/2}}{\sqrt{\pi}\alpha^{k}\Gamma(m+1/2)}\Gamma\Big(m+\frac{k+1}{2}\Big)\Gamma\Big(\frac{k+1}{2}\Big)\,{}_2F_1\bigg(\frac{k+1}{2},m+\frac{k+1}{2};\frac{1}{2};\frac{\beta^2}{\alpha^2}\bigg).
\end{equation}

\begin{proof} By independence,  $\mathbb{E}[|Z|^k]  = \mathbb{E}[|X|^k] \mathbb{E} [ |Y |^{-k}]$, and calculating the absolute moments $\mathbb{E}[|X|^k]$ and $\mathbb{E}[|Y|^{-k}]$ using equation (\ref{mom1}) yields the desired formula for $\mathbb{E}[|Z|^k]$ following a simplification using the identity $\Gamma((k+1)/2)\Gamma((1-k)/2)=\pi \mathrm{sec}(k\pi /2)$.
\end{proof}

\subsection{Product of correlated zero mean normal random variables}

Let $(U_i,V_i)$, $i=1,2$, be independent bivariate normal random vectors with  zero mean vector, variances $(\sigma_{U_i}^2,\sigma_{V_i}^2)$ and correlation coefficient $\rho_i$. Set $s_i=\sigma_{U_i}\sigma_{V_i}$, and let $W_i=U_iV_i$.   It was noted by \cite{gaunt thesis,gaunt prod} that the product $W_i$ has a VG distribution,
\begin{equation}\label{vgrep}W_i\sim\mathrm{VG}\bigg(0,\frac{1}{s_i(1-\rho_i^2)},\frac{\rho_i}{s_i(1-\rho_i^2)},0\bigg),
\end{equation}
 which yielded a new derivation of the PDF of the product $W_i$ which was earlier obtained independently by \cite{gr} and \cite{np16}. Let $T=W_{1}/W_{2}$. On combining (\ref{eq:3}) and (\ref{hlog2}) with (\ref{vgrep}) we immediately deduce the following formulas for the PDF of the ratio $T$. 

\begin{corollary}\label{cor2.10} Let the previous notations prevail. Then, for $t\in\mathbb{R}$,
\begin{align}\label{log}
f_T(t) & =  \frac{ (s_1/s_2)(1-\rho_1^2)^{3/2} t^{-2}}{ \pi^2 \sqrt{1-\rho_2^2}}
  \sum^\infty_{i,j=0} \frac{a_{ij}((i+j)/2)!)^4}{i!j!(i+j+1)!}(2\rho_1)^i\bigg(\frac{1-\rho_1^2}{1-\rho_2^2}\bigg)^j\bigg(\frac{2s_1\rho_2}{s_2t}\bigg)^j \nonumber
 \\ & \quad\times   {}_2F_1\bigg(1+\frac{i+j}{2},1+\frac{i+j}{2};2+i+j; 1-\bigg(\frac{s_1(1-\rho_1^2)}{s_2(1-\rho_2^2) t}\bigg)^2
 \bigg).
\end{align}
The PDF (\ref{log}) simplifies as follows when $\rho_1=\rho_2=0$. For $t\in\mathbb{R}$,
\begin{equation}\label{hlog29}f_T(t)=\frac{2s_2}{\pi^2 s_1}\frac{\log|s_2t/s_1|}{(s_2t/s_1)^2-1}.
\end{equation}
\end{corollary}

\begin{remark}When $s_1=s_2=1$,  (\ref{hlog29}) is also the density of the product of two independent standard Cauchy random variables with PDF $f(x)=\pi^{-1}(1+x^2)^{-1}$, $x\in\mathbb{R}$. This follows from
the well-known result that the quotient of independent standard normal random variables follows the standard Cauchy distribution. More generally, formulas for the PDF of the product of independent Student's $t$ random variables are given in \cite{nad}.
\end{remark}

\appendix

\section{Special functions}\label{appa}
In this appendix, we define the modified Bessel function of the second kind, the Gaussian hypergeometric function and the Meijer $G$-function, and state some basic properties that are needed in this paper. Unless  otherwise stated, these and further properties can be found in the standard references \cite{gradshetyn,luke,olver}. 


The \emph{modified Bessel function of the second kind} is defined, for $m\in\mathbb{R}$ and $x>0$, by
\[K_m(x)=\int_0^\infty \mathrm{e}^{-x\cosh(t)}\cosh(m t)\,\mathrm{d}t.
\]
For $m=0,1,2,\ldots$, we have the elementary representation
\begin{equation}\label{special} K_{m+1/2}(x)=\sqrt{\frac{\pi}{2x}}\sum_{j=0}^m\frac{(m+j)!}{(m-j)!j!}(2x)^{-j}\mathrm{e}^{-x}.
\end{equation}


The \emph{Gaussian hypergeometric function} is defined by the power series
\begin{equation}
\label{gauss}
{}_2F_1(a,b;c;x)=\sum_{j=0}^\infty\frac{(a)_j(b)_j}{(c)_j}\frac{x^j}{j!},
\end{equation}
for $|x|<1$, and by analytic continuation elsewhere. Here $(u)_j=u(u+1)\cdots(u+k-1)$ is the ascending factorial. 

A special case is the following:
\begin{equation}\label{hlog} {}_2F_1(1,1;2;x)=-x^{-1}\log(1-x).
\end{equation}
The hypergeometric function has the following behaviour as $x\rightarrow 1^{-}$.

\noindent{1. $c-a-b>0$:}
\begin{align}\label{f1}{}_2F_1(a,b;c;1)=\frac{\Gamma(c)\Gamma(c-a-b)}{\Gamma(c-a)\Gamma(c-b)}.
\end{align}
\noindent{2. $c-a-b=0$:}
\begin{align}\label{f2}{}_2F_1(a,b;a+b;x)\sim-\frac{\Gamma(a+b)}{\Gamma(a)\Gamma(b)}\log(1-x), \quad x\rightarrow1^{-}.
\end{align}
\noindent{3. $c-a-b<0$:}
\begin{align}\label{f3}{}_2F_1(a,b;c;x)\sim\frac{\Gamma(c)\Gamma(a+b-c)}{\Gamma(a)\Gamma(b)}(1-x)^{c-a-b}, \quad x\rightarrow1^{-}.
\end{align}


The \emph{Meijer $G$-function} is defined, for $x\in\mathbb{R}$, by the contour integral
\begin{equation}\label{mdef}G^{m,n}_{p,q}\bigg(x \, \bigg|\, {a_1,\ldots, a_p \atop b_1,\ldots,b_q} \bigg)=\frac{1}{2\pi i}\int_L\frac{\prod_{j=1}^m\Gamma(b_j-s)\prod_{j=1}^n\Gamma(1-a_j+s)}{\prod_{j=n+1}^p\Gamma(a_j-s)\prod_{j=m+1}^q\Gamma(1-b_j+s)}x^s\,\mathrm{d}s,
\end{equation}
where the integration path $L$ separates the poles of the factors $\Gamma(b_j-s)$ from those of the factors $\Gamma(1-a_j+s)$, and we use the convention that the empty product is $1$.

The $G$-function satisfies the relations
\begin{align}\label{meijergidentity0}G_{p,q}^{m,n}\bigg(x \; \bigg| \;{a_1,\ldots,a_p \atop b_1,\ldots,b_q}\bigg)&=G_{q,p}^{n,m}\bigg(x^{-1} \, \bigg| \,{1-b_1,\ldots,1-b_q \atop 1-a_1,\ldots,1-a_p}\bigg),\\
\label{meijergidentity}x^\alpha G_{p,q}^{m,n}\bigg(x \; \bigg| \;{a_1,\ldots,a_p \atop b_1,\ldots,b_q}\bigg)&=G_{p,q}^{m,n}\bigg(x \, \bigg| \,{a_1+\alpha,\ldots,a_p+\alpha \atop b_1+\alpha,\ldots,b_q+\alpha}\bigg).
\end{align}
The hypergeometric function can be expressed in terms of the Meijer $G$-function:
\begin{equation}\label{2f1m}{}_2F_1(a,b;c;1-x)=\frac{\Gamma(c)}{\Gamma(a)\Gamma(b)\Gamma(c-a)\Gamma(c-b)}G_{2,2}^{2,2}\bigg(x \, \bigg| \,{1-a,1-b \atop 0,c-a-b}\bigg).
\end{equation}
(see \texttt{http://functions.wolfram.com/07.23.26.0007.01}). Also, on combining equations 5.4(1) and 5.4(13) of \cite{luke}, we obtain the indefinite integral formula
\begin{equation}\label{mint}\int x^{\alpha-1}G_{p,q}^{m,n}\bigg(x \; \bigg| \;{a_1,\ldots,a_p \atop b_1,\ldots,b_q}\bigg)\,\mathrm{d}x=x^\alpha G_{p+1,q+1}^{m,n+1}\bigg(x \, \bigg| \,{1-\alpha,a_1,\ldots,a_p \atop b_1,\ldots,b_q,-\alpha}\bigg).
\end{equation}

\section*{Acknowledgements}
SL is supported by a University of Manchester Research Scholar Award.

\end{document}